\documentclass[12pt]{amsart}

\usepackage{amssymb}

\usepackage{enumitem}

\makeatletter
\@namedef{subjclassname@2010}{%
  \textup{2010} Mathematics Subject Classification}
\makeatother

\newtheorem{theorem}{Theorem}[section]
\newtheorem{corollary}[theorem]{Corollary}
\newtheorem{lemma}[theorem]{Lemma}
\newtheorem{proposition}[theorem]{Proposition}

\theoremstyle{definition}
\newtheorem{definition}[theorem]{Definition}
\newtheorem{remark}[theorem]{Remark}
\newtheorem{example}[theorem]{Example}

\numberwithin{equation}{section}

\newcommand{\N}{\mathbb{N}}
\newcommand{\R}{\mathbb{R}}
\newcommand{\Q}{\mathbb{Q}}
\newcommand{\sub}{\subseteq}
\newcommand{\ssub}{\sqsubseteq}
\newcommand{\e}{\epsilon}
\renewcommand{\d}{\delta}
\newcommand{\es}{\emptyset}
\newcommand{\sT}{\mathcal{T}}
\newcommand{\sC}{\mathcal{C}}
\newcommand{\supp}{\operatorname{supp}}

\begin{document}

\baselineskip=17pt

%%%%%%%%%%%%%%%%

\title[Good Sequences]{Schauder bases having many good block basic sequences}

\author[C. A. Krause]{Cory A. Krause}
\address{LeTourneau University\\ 2100 S Mobberly\\
Longview, TX 75602\\
USA}
\email{corykrause@letu.edu}

\date{}

\begin{abstract}
In the study of asymptotic geometry in Banach spaces, a basic sequence which gives rise to a spreading model has been called a good sequence. It is well known that every normalized basic sequence in a Banach space has a subsequence which is good. We investigate the assumption that every normalized block tree relative to a basis has a branch which is good. This combinatorial property turns out to be very strong and is equivalent to the space being $1$-asymptotic $\ell_p$ for some $1\leq p\leq\infty$. We also investigate the even stronger assumption that every block basic sequence of a basis is good. Finally, using the Hindman-Milliken-Taylor theorem, we prove a stabilization theorem which produces a basic sequence all of whose normalized constant coefficient block basic sequences are good, and we present an application of this stabilization.
\end{abstract}



\maketitle

\section{Introduction and Preliminaries}
Let $A\sub\N$ be an infinite set and $[A]^\lambda$ represent the collection of all subsets of $A$ having size $\lambda$. A seminal combinatorial theorem of F. P. Ramsey from \cite{Ram} is the following:

\begin{theorem}[Ramsey]\label{Thrm:Ramsey}
	Let $k,r\in\N$. For all $f:[\N]^k\to\{1,\ldots,r\}$ there exists an $M\in[\N]^\omega$ such that $f$ restricted to $[M]^k$ is constant.
\end{theorem}
	
The use of so-called Ramsey theory has become common in Banach space theory. Here, as in other areas of analysis, the finite set $\{1,\ldots,k\}$ is often replaced by a compact metric space, and equality is replaced by approximate equality. Perhaps the first such use of this theorem in Banach spaces was by Brunel and Sucheston in \cite{BruSuc} to prove the existence of spreading models. Recall that a basic sequence in a Banach space is $C$-spreading if it is $C$-equivalent to all of its subsequences.

\begin{definition}\label{Def:spreadingmodel}
	Let $(e_i)$ be a normalized $1$-spreading basis for a Banach space $(E,\|\cdot\|_s)$  and $(x_i)$ be a normalized basic sequence in a Banach space $(X,\|\cdot\|)$. Assume that for all $\e>0$ and $n\in\N$ there exists $N\in\N$ so that for all scalars $(a_i)_{i=1}^n\in[-1,1]^n$ and all increasing naturals $N\leq k_1<k_2<\cdots<k_n$ we have
	\[
	\Bigg| \Big\|\sum_{i=1}^n a_i x_{k_i}\Big\|-\Big\|\sum_{i=1}^n a_i e_i\Big\|_s \Bigg|\leq\e.
	\]
	Then $(e_i)$ is said to be the \emph{spreading model} generated by $(x_i)$. We  also say that $E$ is a spreading model of $X$ with respect to the sequence $(x_i)$.
\end{definition}
 
 A given space $X$ may have many different spreading models with respect to different sequences $(x_i)$ in $X$. Additionally, an arbitrary normalized basic sequence $(x_i)$ in $X$ usually does not generate any spreading model. In order for the sequence $(x_i)$ to have a spreading model, it is necessary (and sufficient) that the limits $\lim_{k_1\to\infty} \|\sum_{i=1}^n a_i x_{k_i}\|$ exist for all finite sequences of scalars $(a_i)$. Such a sequence in a Banach space has been called a \emph{good} sequence. Brunel and Sucheston showed the following theorem, the proof of which will be useful to recall.
 
 \begin{theorem}[Brunel-Sucheston]\label{Thrm:spreadingmodel}
 	Let $(x_i)$ be a normalized basic sequence in a Banach space $X$. Then some subsequence of $(x_i)$ is good.
 \end{theorem}

\begin{proof} 
	Fix some sequence $(\e_m)$ of numbers such that $0<\e_m<1$ and $\e_m\to 0$. Let $(a_i)_{i=1}^n\in[-1,1]^n$ be scalars. Define the map $f:[\N]^n\to[0,n]$ by $f(k_1,\ldots,k_n)=\|\sum_{i=1}^n a_i x_{k_i}\|$ where $(k_i)$ are in increasing order. Since $[0,n]$ is compact, we apply Ramsey's theorem to find an $M_1\in[\N]^\omega$ such that the subsequence $(y_i^1)$ of $(x_i)$ corresponding to $M_1$ satisfies 
	\[
	\left|\,\Big\|\sum_{i=1}^n a_i y_{j_i}^1\Big\|-\Big\|\sum_{i=1}^n a_i y_{k_i}^1\Big\|\,\right|\leq\e_1
	\]
	for any $j_1<\ldots<j_n$ and $k_1<\ldots<k_n$. Repeating this process for each $m$, we can find nested subsets $M_m$ corresponding to nested subsequences $(y_i^m)$ such that
	\[
	\left|\,\Big\|\sum_{i=1}^n a_i y_{j_i}^m\Big\|-\Big\|\sum_{i=1}^n a_i y_{k_i}^m\Big\|\,\right| \leq\e_m
	\]
	for any $j_1<\ldots<j_n$ and $k_1<\ldots<k_n$. Let $(y_i)$ be the diagonal sequence of these $(y_i^m)$. Then the limit	$\lim_{k_1\to\infty} \|\sum_{i=1}^n a_i y_{k_i}\|$ exists.
	
	Now enumerate the countable set
	\[
	R=\{(a_i)_{i=1}^n: n\in\N, \, a_i\in\Q\cap[-1,1] \text{ for } 1\leq i\leq n\}.
	\]
	By the technique of the previous paragraph, we can find nested subsequences $(z_i^m)$ of $(x_i)$ such that the required limit exits whenever any of the scalars coming from the first $m$ elements in the enumeration of $R$ are used. Letting $(z_i)$ be the diagonal sequence of the $(z_i^m)$, the required limit exists whenever any $(a_i)$ in $R$ is used showing that $(z_i)$ is a good subsequence of $(x_i)$.
\end{proof}

There are many theorems in Banach space theory having the above form, namely, that every sequence in $X$  of a certain type has a subsequence satisfying a property $P$. A different assumption is also sometimes investigated, namely, that every tree in $X$ with certain specified structure has a branch satisfying a property $P$. Often, for a given property $P$, this branch condition is significantly stronger than the subsequence condition. For example, one of the first investigations along this line is \cite{OdeSch1}. There it is shown that, if every tree on $X$ has a branch with a certain property, then $X$ can be embedded into a space with an FDD $(E_i)$ so that all normalized sequences in $X$ which are almost a skipped blocking of $(E_i)$ have that property. This is then used to show that, if $X$ is a separable reflexive Banach space such that every normalized weakly null tree on $X$ has a branch uniformly equivalent to the unit vector basis of $\ell_p$, then $X$ isomorphically embeds into an $\ell_p$ sum of finite dimensional spaces. Branches and trees have also been used to characterize certain classical properties of Banach spaces such as having a separable dual (see \cite{DutFon}). 

We present some of the necessary definitions regarding trees and branches in Banach spaces.

\begin{definition}\label{Def:treeonX}
	A (countably splitting) \emph{tree} $\sT$ \emph{on a Banach space} $X$ is a function from $[\N]^{<\omega}$ to $X\cup\{\es\}$, written $\sT=(x_A)_{A\in[\N]^{<\omega}}$, where $x_A=\es$ iff $A=\es$. An element $x_A$ in $\sT$ is called a \emph{node}. The ordering of the elements of $\sT$ is the ordering induced by the usual initial segment partial ordering on $[\N]^{<\omega}$ where $A\preceq B$ iff there exists an $n\in\N$ such that $B\cap\{1,\ldots,n\}=A$.
\end{definition}

Through this induced ordering, relevant definitions concerning trees on $\N$ carry over to trees on $X$. In particular, a branch of $\sT$ is a sequence $(x_{\{k_1,\ldots,k_n\}})_{n=1}^\infty$ in $X$ where $k_1<k_2<\ldots$ is an increasing sequence of naturals. If $x_A$ is a node in a countably splitting tree on $X$, the collection of successors of $x_A$ is naturally ordered by considering $x_{A\cup\{n\}}<x_{A\cup\{m\}}$ iff $n<m$. We will call this sequence \emph{the sequence of nodes below} $A$.

We call $\sT$ \emph{normalized} iff every node of $\sT$ is normalized. Assume that $(x_i)$ is a basic sequence in $X$. Then $\sT$ is a \emph{block tree} (with respect to $(x_i)$) iff, for every $A\in[N]^{<\omega}$, the sequence of nodes below $A$ is a block basic sequence of $(x_i)$. Notice that, if $\sT$ is a block tree on $X$, then there is a subtree of $\sT$ so that all the branches of the subtree are block basic sequences.

\begin{remark} Let $(x_i)$ be a sequence in a Banach space $X$. Then the tree $(x_A)$ on $X$ defined by $x_A=x_{\max A}$  is the tree of partial subsequences of $(x_i)$. The collection of branches of $(x_A)$ is precisely the collection of subsequences of $(x_i)$. This demonstrates that what we call the ``branch condition'' above is at least as strong as what we call the ``subsequence condition.''
\end{remark}

The generality of Theorem \ref{Thrm:spreadingmodel} leads us to investigate the strength of the corresponding branch condition. Theorem \ref{Thrm:1asymplp}, Corollary \ref{Cor:containlp}, and Theorem \ref{Thrm:converse1} will show:

\vspace{.15in}
\noindent\textbf{Theorem A.} \emph{Let $X$ be a Banach space with basis $(x_i)$. Then every normalized block tree on $X$ has a good branch iff $X$ is $1$-asymptotic $\ell_p$ for some $1\leq p\leq\infty$. In particular, such a space contains almost isometric copies of $\ell_p$ (or $c_0$ in the case $p=\infty$).}
\vspace{.15in}

The main tool in this regard will be Krivine's theorem (see \cite{Kri} and \cite{Lem}) which we state here in a convenient form for our use.

\begin{theorem}[Krivine, Lemberg]\label{Thrm:krivine}
	Let $(x_i)$ be a normalized basis for a Banach space $X$. There exists a $1\leq p\leq\infty$ such that for all $n\in N$ and $\e>0$ there exists a block basic sequence $(y_i)$ of $(x_i)$ such that any $n$-vectors of $(y_i)$ are $(1+\e)$-equivalent to the unit vector basis of $\ell_p^n$.
\end{theorem}

We will also investigate the even stronger assumption that all block basic sequences are good. We show that this does not characterize the unit vector bases of $\ell_p$ and $c_0$; however, Theorem \ref{Thrm:stab1asymplp}, Corollary \ref{Cor:lpsum}, and Theorem \ref{Thrm:converse2} will show:

\vspace{.15in}
\noindent\textbf{Theorem B.} \emph{Let $X$ be a Banach space with basis $(x_i)$. Then every block basic sequence $(x_i)$ is good iff $X$ is stabilized $1$-asymptotic $\ell_p$ for some $1\leq p\leq\infty$. In particular, such a space is isomorphic to an $\ell_p$ sum of finite dimensional spaces (or $c_0$ in the case $p=\infty$).}
\vspace{.15in}

The notion of asymptotic structure was developed in \cite{MauMilTom} by Bernard Maurey, Vitali Milman, and Nicole Tomczak-Jaegermann. Like the concept of spreading models, it is a way of describing the asymptotic geometry of a Banach space by looking at the uniform behavior of finite subspaces. Although other filters of infinite dimensional subspaces of $X$ may be used, we will define asymptotic structure relative to the tail subspaces $[x_i]_{i=n}^\infty$ where $(x_i)$ is a fixed basis of $X$. 

\begin{definition}\label{Def:asymptoticstructure}
	Fix $n\in\N$. Let $E$ be an $n$-dimensional Banach space with basis $(e_i)_{i=1}^n$. Then $E$ is called an (n-dimensional) \emph{asymptotic space} for a Banach space $X$ with basis $(x_i)$ provided that for all $\e>0$ we have
	\begin{align*}
	\forall m_1&\in\N \,\exists y_1\in S([x_i]_{i=m_1}^\infty)\quad\cdots\quad\forall m_n\in\N \,\exists y_n\in S([x_i]_{i=m_n}^\infty) \\
	&\left( (y_i)_{i=1}^n \text{ is a block sequence of $(x_i)$ and $(1+\e)$-equivalent to $(e_i)_{i=1}^n$}\right).
	\end{align*}
	The collection of all $n$-dimensional asymptotic spaces for $X$ is denoted $\{X\}_n$. We will often identify $(e_i)$ with $E$ and $(x_i)$ with $X$ by a  slight  and common abuse.
\end{definition}

Note that the above definition is equivalent to postulating the existence of a winning strategy for player two in a finite game of length $n$ where the first player plays natural numbers determining tail subspaces and the second player plays normalized vectors of finite support in the tail subspaces.  The winning condition for player two is that the vectors played form a block basic sequence which is $(1+\e)$-equivalent to $(e_i)_{i=1}^n$. The sequence of moves $(m_1,y_1,m_2,y_2,\ldots,m_n,y_n)$ will be called a \emph{run} of the game, and the block vectors $(y_i)_{i=1}^n$ will be called an \emph{outcome} of the game.

By Krivine's theorem, there exists a $1\leq p\leq\infty$ such that for all $n$ we know $\ell_p^n$ is in $\{X\}_n$. In particular, we know that $\{X\}_n\not=\es$ for all $n$. If these are the only asymptotic spaces for $X$ up to $C$-equivalence, then $X$ is called an asymptotic $\ell_p$ space.

\begin{definition}\label{Def:asymptoticlp}
	A Banach space $X$ is called \emph{asymptotic} $\ell_p$ where $1\leq p\leq\infty$ provided that there exists a $C\geq 1$ such that for all $n\in\N$ and all $E\in\{X\}_n$ we have $E\stackrel{C}{\sim}\ell_p^n$. In this case, we will also call $X$ a $C$-asymptotic $\ell_p$ space.
\end{definition}

Next we state an equivalent way of defining $\{X\}_n$ which was given in \cite{MauMilTom}. First, a preliminary definition:

\begin{definition}\label{Def:subspaceplayer}
	Let $X$ be a Banach spaces with basis $(x_i)$. Given a collection $\sC$ of $n$-dimensional Banach spaces with fixed bases, we say that $\sC$ is \emph{subspace winning} provided that for all $\e>0$ we have
	\begin{align*}
		\exists m_1\in\N &\,\forall y_1\in S([x_i]_{i=m_1}^\infty)\quad\cdots\quad\exists m_n\in\N \,\forall y_n\in S([x_i]_{i=m_n}^\infty) \\
		( &(y_i)_{i=1}^n \text{ is a block basic sequence of $(x_i)$} \\ &\qquad\Rightarrow \text{ $\exists\,[e_i]_{i=1}^n\in\sC \,( (y_i)$ is $(1+\e)$-equivalent to $(e_i)$)}).
	\end{align*}
\end{definition}

Again, we can state the above definition by saying that $\sC$ is subspace winning whenever player one has a winning strategy in our usual game to force the outcome to be close to an element of $\sC$. The equivalent characterization of $\{X\}_n$ from \cite{MauMilTom} is:

\begin{theorem}\label{Thrm:subspaceplayer}
	Let $X$ be a Banach space with basis $(x_i)$. Then 
	\[
	\{X\}_n=\bigcap\{\sC:\sC \text{ is $n$-dimensional subspace winning}\}.
	\]
\end{theorem}

Next, we cite a result we will need about stabilizing asymptotic structure from \cite{MilTom}. It says that there is always a subspace of $X$ where any outcome of the asymptotic game results in a sequence close to a member of $\{X\}_n$ provided that the support of player two's first move is far enough out. 

\begin{theorem}[Milman-Tomczak-Jaegerman]\label{Thrm:MilTom}
	Let $X$ be a Banach space. There exists a subspace $Z$ of $X$ with a normalized basis $(z_i)$ such that the following holds:
	\begin{quote}
		For all $n\in\N$ and $\e>0$ there exists $N=N(n,\e)$ such that for any normalized block vectors $(y_i)_{i=1}^n$ of $(z_i)$ with $\supp(y_1)>N(n,\e)$ there exists $E\in\{X\}_n$ such that $(y_i)_{i=1}^n\stackrel{1+\e}{\sim} E$.
	\end{quote}
	The subspace $Z$ will be called a stabilizing subspace for $X$.
\end{theorem}

Finally, in Theorem \ref{Thrm:newstab} and Theorem \ref{Thrm:asymptoticzippin}, we prove:

\vspace{.15in}
\noindent\textbf{Theorem C.} \emph{Let $X$ be a Banach space. Then there exists a (not necessarily semi-normalized) basic sequence $(x_i)$ in $X$ such that every normalized block basic sequence of $(x_i)$ is good. Additionally, if the sequence $(x_i)$ produced is semi-normalized and also unconditional, then all of the spreading models of these good sequences are uniformly equivalent to the unit vector basis of $c_0$ or $\ell_p$ for some $1\leq p<\infty$.}
\vspace{.15in}
 
In the following definitions, we use the notation $E<F$ to indicate that $\max(E)<\min(F)$ when $E,F\in[\N]^{<\omega}$.

\begin{definition}\label{Def:constantcoeff}
	Let $(x_i)$ be a basic sequence in a Banach space $X$. A block basic sequence $(y_i)$ of $(x_i)$ is a \emph{normalized constant coefficient block basic sequence} (or NCCB basic sequence) iff there exists a sequence of finite subsets of naturals $(E_i)\sub[\N]^{<\omega}$ such that $E_i<E_{i+1}$ and
	\[
	y_i=\frac{\sum_{k\in E_i} x_k}{\|\sum_{k\in E_i} x_k\|}
	\]
	for all $i\in\N$. We will say that the NCCB basic sequence $(y_i)$ corresponds to the sequence $(E_i)$ since each $(y_i)$ is completely determined by $(x_i)$ and $(E_i)$.
\end{definition}

Following a technique in \cite{HalOde}, our work uses a generalization of Ramsey's theorem.  Rather than stabilizing a property of finite subsets of naturals relative to further infinite subsets, this generalization involves stabilizing a property of finite blockings of naturals relative to coarser infinite blockings.

\begin{definition}\label{Def:block}
	A finite or infinite sequence $(E_i)\sub[\N]^{<\omega}$ is called a \emph{blocking of naturals} provided that for all $i$ we have $E_i<E_{i+1}$. In this context, we use $\N$ to denote the blocking of all singletons. If $E=(E_i)$ and $F=(F_i)$ are blockings of naturals, then $F$ is \emph{coarser} than $E$ provided that every element of $F$ is a union of elements of $E$. We denote this by $F\ssub E$. 
\end{definition}

If $P=(P_i)_{i=1}^\infty$ is an infinite blocking of naturals, then we use the notation $\langle P\rangle^\lambda$ to denote the collection of blockings of naturals which are coarser than $P$ and of length $\lambda$.

The first step in obtaining the Milliken-Taylor theorem was given by Neil Hindman in \cite{Hin}. Although Hindman's result was originally stated in terms of finite sums, we give a reformulation in terms of finite unions and our previously defined notation.

\begin{theorem}[Hindman]\label{Thrm:hindman}
	If $r\in\N$ and $f:[\N]^{<\omega}\to\{1,\ldots,r\}$ then there exists $Q\in\langle\N\rangle^\omega$ such that $f$ is constant on the collection of sets which are finite unions of elements from $Q$.
\end{theorem}

Hindman's theorem was later independently used by Keith R. Milliken in \cite{Mil} and Alan D. Taylor in \cite{Tay} to prove:

\begin{theorem}[Milliken, Taylor]\label{Thrm:miltay}
	Let $k,r\in\N$ and $P\in\langle\N\rangle^\omega$. For any $f:\langle P\rangle^k\to\{1,\ldots,r\}$ there exists a $Q\in\langle P\rangle^\omega$ such that $f$ restricted to $\langle Q\rangle^k$ is constant.
\end{theorem}

This theorem clearly generalizes Ramsey's original theorem by taking the functions $f$ in the Milliken-Taylor theorem to be those which depend only on the minimum elements of the $k$ sets forming some $(E_i)\in\langle P\rangle^k$. Notice that if we take $k=1$ in Ramsey's original theorem, we get the so-called pigeonhole principle. Similarly, if we take $k=1$ and $P=\N$ in the Milliken-Taylor theorem, we get Hindman's theorem because $\langle\N\rangle^1=[\N]^{<\omega}$ and $\langle Q\rangle^1$ is the collection of sets which are finite unions from $Q$. In fact, the proof of the Milliken-Taylor theorem follows from Hindman's theorem precisely the same way that Ramsey's theorem follows from the pigeonhole principle.

\section{The Assumption Every Block Tree Has a Good Branch} \label{Sec:GoodBranches}

Throughout the first portion of this section, we fix a Banach space $X$ with a basis $(x_i)$ and assume that $(x_i)$ has the relevant property: every normalized block tree on $X$ with respect to $(x_i)$ has a good branch. The next two propositions will show that such a space has an isometrically unique spreading model. To prove them, it is useful to make the following provisional definition.

\begin{definition}\label{Def:blockfromarray}
	Let $(\mathbf{y}^n)_{n=1}^\infty$ be a 2-dimensional array such that every row $\mathbf{y}^n=(y_i^n)_{i=1}^\infty$ of the array is a block basic sequence of $(x_i)$. We define the \emph{block tree based on} $(\mathbf{y}^n)$ as follows: First $x_{\es}=\es$. Now assume that $x_A$ has been defined for some $A\in[\N]^{<\omega}$ and that $x_A=y_i^n$ for some vector in the array. Let $m\in\N$ be least such that $\supp(y_{m+1}^{n+1})>\supp(x_A)$. Then put, for $k\in\N$, 
	\[
	x_{A\cup\{\max A+k\}}=y_{m + k}^{n+1}.
	\]
	That is, the sequence of nodes below $A$ are all the block vectors from the next row of the array whose supports come after $y^n_i$.
\end{definition}

Since the collection of successors of each node is the tail of one of the block basic sequences $(y^n_i)$, the countably splitting tree $(x_A)$ is indeed a block tree. Furthermore, note that any branch of the tree $(x_A)$ is a block basic sequence of $(x_i)$. 

\begin{proposition}\label{Prop:mustbelp2}
	Assume that $(x_i)$ and $X$ are as stated. Then there exists a good block basic sequence of $(x_i)$ which generates a spreading model 1-equivalent to the unit vector basis of $c_0$ or to $\ell_p$ for some $1\leq p<\infty$.
\end{proposition}

\begin{proof}
	Let $(\e_n)$ be a sequence of positive real numbers converging to zero. By Krivine's Theorem, Theorem \ref{Thrm:krivine}, there exists a $1\leq p\leq \infty$ such that for all $n$ there are block subsequences $\mathbf{y}^n=(y_i^n)$ such that for any $n$-tuple $(i_1,\ldots,i_n)$ we have $(y_{i_k}^n)_{k=1}^n$ is $(1+\e_n)$-equivalent to $\ell_p^n$. Define the new array $(\mathbf{z}^n)$ consisting of the rows $\mathbf{y}^n$ repeating $n$ times; that is, our array $(\mathbf{z}^n)$ consists of the rows 
	\[
	(\mathbf{y}^1,\mathbf{y}^2,\mathbf{y}^2,\mathbf{y}^3,\mathbf{y}^3,\mathbf{y}^3,\mathbf{y}^4,\mathbf{y}^4,\mathbf{y}^4,\mathbf{y}^4,\ldots).
	\]
	Let $(x_A)$ be the block tree based on $(\mathbf{z}^n)$. By our assumption, this tree has a good branch. This branch must have a spreading model 1-equivalent to the unit vector basis of $\ell_p$ (or $c_0$ in the case $p=\infty$). This is because, for a fixed $n$, $n$ many vectors which are $(1+\e_n)$-equivalent to $\ell_p^n$ must appear in the branch infinitely often which implies that the first $n$ many vectors of the spreading model are $(1+\e_n)$-equivalent to the unit vector basis of $\ell_p^n$. Since $n$ is arbitrary, the result follows.
\end{proof}

\begin{proposition}\label{Prop:mustbelp3}
	Assume that $(x_i)$ and $X$ are as stated. Then there exists a $1\leq p\leq \infty$ such that every spreading model of a block basic sequence of $(x_i)$ must be 1-equivalent to the unit vector basis of $\ell_p$  (or $c_0$ in the case $p=\infty$).
\end{proposition}

\begin{proof}
	Let $1\leq p\leq\infty$ be as in Proposition \ref{Prop:mustbelp2}. We focus on the case $p<\infty$ and the proof for $p=\infty$ is similar. Let $(y_i)$ be a normalized block basic sequence of $(x_i)$ which has a spreading model 1-equivalent to the unit vector basis of $\ell_p$. Suppose there is some other good block basic sequence $(z_i)$ of $(x_i)$ whose spreading model is not 1-equivalent to the unit vector basis of $\ell_p$. This means that there exists an $\e>0$, a number $m\in\N$, and scalars $(a_i)_{i=1}^m$  such that $(\sum_{i=1}^m |a_i|^p)^{1/p}=1$ but
	\[
	\lim_{k_1\to\infty}\Big\|\sum_{i=1}^m a_i z_{k_i}\Big\|\geq1+\e \quad \text{or} \quad \lim_{k_1\to\infty}\Big\|\sum_{i=1}^m a_i z_{k_i}\Big\|\leq1-\e.
	\]
	Fix such an $\e$, $m$, and $(a_i)$. Define the array $(\mathbf{w}^n)$ by 
	\[
	\mathbf{w}^n = \left\{
	\begin{array}{ll}
	(y_i) & \text{if there is a $k$ even with $km+1\leq n\leq (k+1)m$} \\
	(z_i) & \text{if there is a $k$ odd with $km+1\leq n\leq (k+1)m$} 
	\end{array}
	\right.
	\]
	In other words, $(\mathbf{w}^n)$ consists of $m$ many rows of $(y_i)$, then $m$ many rows of $(z_i)$, then $m$ many rows of $(y_i)$, and so on. Let $(x_A)$ be the block tree based on the array $(\mathbf{w}^n)$. We claim this block tree has no good branches, a contradiction to our assumption concerning $(x_i)$. Let $(v_i)$ be a normalized block basic sequence of $(x_i)$ which is a branch of $(x_A)$. Then we have that for infinitely many increasing $m$-tuples $(k_1,\ldots,k_m)$ 
	\[
	1-\frac{\e}{2}\leq \| \sum_{i=1}^m a_i v_{k_i} \|\leq 1+\frac{\e}{2}
	\]
	and that for infinitely many increasing $m$-tuples $(k_1,\ldots,k_m)$ 
	\[
	\Big\| \sum_{i=1}^m a_i v_{k_i} \Big\|\geq1+\e \quad \text{or} \quad \Big\| \sum_{i=1}^m a_i v_{k_i} \Big\|\leq1-\e
	\]
	which implies that the limit $\lim_{k_1\to\infty}\| \sum_{i=1}^m a_i v_{k_i} \|$ cannot exist and $(v_i)$ is not good.
\end{proof}

\begin{remark}\label{Rem:spreadingmodelproblem}
In \cite{OdeSch2}, Odell and Schlumprecht showed that if $c_0$ or $\ell_1$ is the isometrically unique spreading model of a Banach space $X$, then $X$ contains an isomorphic copy of $c_0$ or $\ell_1$ (respectively). The analogous question for $1<p<\infty$ is open. In our case, where every normalized block tree has good branch, the existence of almost isometric copies of $c_0$ or $\ell_p$ in $X$ will follow in Corollary \ref{Cor:containlp}. 
\end{remark}

\begin{theorem}\label{Thrm:1asymplp}
	Assume that $(x_i)$ and $X$ are as stated. Let $1\leq p\leq \infty$ be as in Proposition \ref{Prop:mustbelp3}. Then $X$ is $1$-asymptotic $\ell_p$.
\end{theorem}

\begin{proof}
	Let $(y_i)$ be any good block basic sequence of $(x_i)$ which generates a spreading model $1$-equivalent to the unit vector basis of $\ell_p$. Suppose that $X$ is not $1$-asymptotic $\ell_p$. Then there exists an $m\in\N$ and $\e>0$ so that some $(e_i)_{i=1}^m\in\{X\}_n$ is not $(1+\e)$ equivalent to $\ell_p^m$. Fix a strategy for player two in the game associated with $(e_i)$ and $\e/2$. We define a normalized block tree $(x_A)$ by defining it on each level; that is, by induction on $|A|$. This block tree will be such that any branch through it has the form
	\begin{equation}\label{Eq:form}
	(y_{k_1},\ldots,y_{k_m},z_1^1,\ldots,z_m^1,y_{k_{m+1}},\ldots,y_{k_{2m}},z_1^2,\ldots,z_m^2,y_{k_{2m+1}},\ldots)
	\end{equation}
	where $(y_{k_i})$ is a subsequence of $(y_i)$ and where for all $k\in\N$ $(z_i^k)_{i=1}^m$ is the outcome of a game where player two uses the strategy above.
	
	Let $x_\es=\es$ as usual, and then assume that $x_A$ has been defined. 
	First, consider the case where the successors of $x_A$ will be on the $n$th level where $n$ satisfies  $\exists k \text{ even } (km+1\leq n\leq (k+1)m)$. In this case, let $p$ be least such that $\supp(y_{p+1})>\supp(x_A)$ and then put $x_{A\cup\{\max A +k\}}=y_{p+k}$
	for $k\in\N$. Next, consider the case where the successors of $x_A$ will be on the $n$th level where $n$ satisfies $\exists k \text{ odd } (km+1\leq n\leq (k+1)m)$. If $n=km+1$ specifically, then define $x_{A\cup\{\max A+1\}}$ to be player two's response when player one begins with the move $\max(\supp x_A)+1$. Then define $x_{A\cup\{\max A+k\}}$ to be player two's response when player one begins with the move $\max(\supp x_{A\cup\{\max A +k-1\}})+1$. If $n>km+1$, then we may inductively assume the presence of a partial run of the game ending with the move $x_A$ played by player two.  Define $x_{A\cup\{\max A+1\}}$ to be player two's response when player one continues the run of the game with the move $\max(\supp x_A)+1$. Next, define $x_{A\cup\{\max A+k\}}$ to be player two's response when player one continues the run of the game with $\max(\supp x_{A\cup\{\max A +k-1\}})+1$. This defines $(x_A)$ so that every branch through $(x_A)$ will have the form of (\ref{Eq:form}).
	
	We claim that $(x_A)$ has no good branches, a contradiction to our assumption concerning $(x_i)$. Let $(v_i)$ be a normalized block basic sequence of $(x_i)$ which is a branch of $(x_A)$. Since $(e_i)$ is not $(1+\e)$-equivalent to $\ell_p^n$, there exists scalars $(a_i)$ such that $(\sum_{i=1}^m |a_i|^p)^{1/p}=1$, but
	\[
	\Big\| \sum_{i=1}^m a_i e_{i} \Big\|>1+\e \quad \text{or} \quad \Big\| \sum_{i=1}^m a_i e_{i} \Big\|<\frac{1}{1+\e}.
	\]
	By form (\ref{Eq:form}), there are infinitely many increasing $m$-tuples $(k_1,\ldots,k_m)$ such that $(v_{k_1},\ldots,v_{k_m})$ is $1+\e/2$ equivalent to $(e_i)$ and so infinitely increasing $m$-tuples $(k_1,\ldots,k_m)$ such that
	\[ 
	\Big\| \sum_{i=1}^m a_i v_{k_i} \Big\|>\frac{1+\e}{1+\e/2}\geq1+\d \quad \text{or} \quad \Big\| \sum_{i=1}^m a_i v_{k_i} \Big\|<\frac{1+\e/2}{1+\e}\leq 1-\d
	\]
	for some $\d>0$. Also, by form (\ref{Eq:form}), there are infinitely many successive increasing $m$-tuples $(k_1,\ldots,k_m)$ such that $(v_{k_1},\ldots,v_{k_m})$ is a subsequence of $(y_i)$ and so infinitely many increasing $m$-tuples $(k_1,\ldots,k_m)$ such that
	\[
	1-\d/2\leq\Big\| \sum_{i=1}^m a_i v_{k_i} \Big\|\leq 1+\d/2
	\]
	because the spreading model of $(y_i)$ is $1$-equivalent to the unit vector basis $\ell_p$. It follows that the limit $\lim_{k_1\to\infty}\| \sum_{i=1}^m a_i v_{k_i} \|$ cannot exist and that $(v_i)$ is not good. Since this is a contradiction, we conclude that $X$ must be $1$-asymptotic $\ell_p$.	
\end{proof}

By applying Theorem \ref{Thrm:MilTom}, one may show that a $1$-asymptotic $\ell_p$ space contains almost isometric copies of $\ell_p$. This is done in \cite{MilTom}. Hence, we obtain the following corollary.

\begin{corollary}\label{Cor:containlp}
	Assume that $(x_i)$ and $X$ are as stated. Let $1\leq p\leq \infty$ be as in Proposition \ref{Prop:mustbelp3}. Then, for all $\e>0$, $X$ contains a sequence which is $(1+\e)$-equivalent to the unit vector basis of $\ell_p$ (or $c_0$ in the case $p=\infty$).
\end{corollary}

The converse of Theorem \ref{Thrm:1asymplp} also holds. For this direction, we will use the characterization given by Theorem \ref{Thrm:subspaceplayer}.

\begin{theorem}\label{Thrm:converse1}
	Assume that $X$ with basis $(x_i)$ is $1$-asymptotic $\ell_p$ for some $1\leq p\leq\infty$. Then every normalized block tree on $X$  with respect to $(x_i)$ has a good branch.
\end{theorem}

\begin{proof}
	 Let $(x_A)_{A\in[\N]^{<\omega}}$ be a normalized block tree on $X$ with respect to $(x_i)$. We will inductively define the sequence of naturals $(k_i)$ so that the sequence $(x_{\{k_1,\ldots,k_n\}})_{n=1}^\infty$ is a good branch of $(x_A)$. Since $X$ is $1$-asymptotic $\ell_p$, Theorem \ref{Thrm:subspaceplayer} implies that for all $n\in\N$ we have winning strategies for player one to produce outcomes $(y_1,\ldots,y_n)$ which are $(1+\frac{1}{n})$-equivalent to $\ell_p^n$. To begin, let $k_1$ be player one's first move in the (trivial) strategy to produce $(y_1)$ which is $2$-equivalent to $\ell_p^1$.
	
	Assume that $k_i$ has been defined for $i<n$. Also inductively assume that for every $m<n$ and every increasing $m$-tuple $(k_{i_1},\ldots,k_{i_m})$ where $i_m<n$ and $m<i_1$ we have a partial run $(k_{i_1},x_{\{k_{i_1}\}},\ldots,k_{i_m}, x_{\{k_{i_1},\ldots,k_{i_m}\}})$ of the game where player one is following a strategy which will produce outcomes $(y_i)_{i=1}^{i_1}$ of length ${i_1}$ which are $(1+\frac{1}{i_1})$-equivalent to $\ell_p^{i_1}$. Since there are only finitely many such partial runs, let $M$ be the maximum of all of player one's responses in each of these games. Also, let $N$ be the first move of player one according to a strategy which will produce an outcome of length $n$ which is $(1+\frac{1}{n})$-equivalent to $\ell_p^n$. We now define $k_n$ to be the least number so that $\supp(x_{\{k_1,\ldots,k_n\}})\geq \max\{N,M\}$ and hence the least $k_n$ so that the next node in the tree is an appropriate move for player two in all of the games referenced above.
	
	Let $(y_i)$ be the block basic sequence $(x_{\{k_1,\ldots,k_n\}})_{n=1}^\infty$, the branch which is determined by the sequence $(k_i)$ inductively defined above. We show that $(y_i)$ is good. For this, let $(a_i)_{i=1}^n$ be scalars and $\e>0$. Let $\d=\e/(\sum_{i=1}^n|a_i|^p)^{1/p}$. By construction of $(y_i)$, any vectors $(y_{k_i})_{i=1}^n$ where $k_1\geq\max\{n,\frac{1}{\d}\}$ are the initial part of an outcome where player one has played according to a strategy which produces $k_1$ vectors which are $(1+\frac{1}{k_1})$-equivalent to the unit vector basis of $\ell_p^{k_1}$. In particular $(y_{k_i})_{i=1}^n$ is $(1+\d)$-equivalent to the unit vector basis of $\ell_p^n$. Thus, for such vectors we have
	\[
	\Big\|\sum_{i=1}^n a_i y_{k_i}\Big\|\leq(1+\d)\left(\sum_{i=1}^n|a_i|^p\right)^{1/p}\leq\left(\sum_{i=1}^n|a_i|^p\right)^{1/p}+\e
	\]
	and
	\begin{align*}
	\Big\|\sum_{i=1}^n a_i y_{k_i}\Big\|&\geq\frac{1}{1+\d}\left(\sum_{i=1}^n|a_i|^p\right)^{1/p} \\ & \geq(1-\d)\left(\sum_{i=1}^n|a_i|^p\right)^{1/p}\geq\left(\sum_{i=1}^n|a_i|^p\right)^{1/p}-\e.
	\end{align*}
	Therefore, $\lim_{k_1\to\infty}\|\sum_{i=1}^n a_i y_{k_i}\|=\left(\sum_{i=1}^n|a_i|^p\right)^{1/p}$ exists, and $(y_i)$ is good.
\end{proof}

\section{The Assumption Every Block Basic Sequence is Good} \label{Sec:AllGood}

In this section, we examine the even stronger assumption that $(x_i)$ is a basis for $X$ such that every normalized block basic sequence of $(x_i)$ is good. Based on the strength of Theorem \ref{Thrm:1asymplp}, one might think that this stronger assumption would actually characterize the unit vector bases of $c_0$ and $\ell_p$. However, this is not quite the case.

\begin{example}\label{Ex:allgood}
	Let $1<p\leq2$ and $(p_s)\sub[1,p)$ be an increasing sequence of real numbers converging up to $p$. Let $(n_s)$ be an increasing sequence of natural numbers. Consider the $\ell_p$ sum of the finite dimensional spaces $\ell_{p_{s}}^{n_s}$
	\[
	X=\left(\sum_{s=1}^\infty \ell_{p_s}^{n_s}\right)_p
	\]
	with the standard basis $(e_{i,j})_{i=1}^\infty\,_{j=1}^{n_i}$, ordered lexicographically, where for each $i\in\N$ the sequence $(e_{i,j})_{j=1}^{n_i}$ is the unit vector basis of $\ell_{p_i}^{n_i}$. We relabel this basis as simply $(f_i)$.
	
	\noindent\emph{Claim 1: Every normalized block basic sequence of $(f_i)$ is good.}
	
	Let $(y_i)$ be a normalized block basic sequence of $(f_i)$. Say, $y_i=\sum_{l\in E_i} b_l f_l$ where $(E_i)$ is a blocking of naturals. For each $i\in\N$, assume that $(G_r)_{r\in B_i}$ partitions $E_i=\supp(y_i)$ so that each $G_r$ is the portion of $\supp(y_i)$ which is in some $\ell_{p_{s}}^{n_{s}}$. Since $(y_i)$ is normalized, we have
	\begin{align}\label{Eq:norm1}
	\|y_i\|=\Big\|\sum_{l\in E_i} b_l f_l\Big\|=\Big\|\sum_{r\in B_i}\sum_{l\in G_r} b_l f_l\Big\| =\left(\sum_{r\in B_i}\Big(\sum_{l\in G_r} |b_l|^{p_{s_r}}\Big)^{p/p_{s_r}}\right)^{1/p}=1.
	\end{align}
	Let $(a_i)_{i=1}^n\in[-1,1]^n$ be scalars. To show that $(y_i)$ is good we must show the existence of  $\lim_{k_1\to\infty}\|\sum_{i=1}^n a_i y_{k_i}\|$. To do this, we show that
	\[
	\sum_{i=1}^n |a_i|^p\leq\Big\|\sum_{i=1}^n a_i y_{k_i}\Big\|^p\leq n^{\frac{p}{p_{s_0}}-1}\Big(\sum_{i=1}^n |a_i|^p\Big)
	\]
	whenever $\min(\supp(y_{k_1}))$ is in the portion of the basis corresponding to $\ell_{p_{s_0}}^{n_{s_0}}$. Since $n$ is fixed and $p_{s_0}\to p$ as $k_1\to\infty$, this shows that 
	\[
	\lim_{k_1\to\infty}\Big\|\sum_{i=1}^n a_i y_{k_i}\Big\|= \Big(\sum_{i=1}^n |a_i|^p\Big)^{1/p}. 
	\]	
	Fix $k_1<\cdots<k_n$. Let $x=\sum_{i=1}^n a_i y_{k_i}$. We have the partition of $\supp(x)$ into $\{E_{k_1},\ldots,E_{k_n}\}$. We may also partition $\supp(x)$ into $\{F_1,\ldots, F_m\}$ where each $F_j$ is the portion of $\supp(x)$ which is in some $\ell_{p_{s}}^{n_{s}}$. Then let $\{G_1,\ldots,G_N\}$ be the common refinement of these two partitions. For each $1\leq j\leq m$ there exists an $A_j\sub\{1,\ldots,N\}$ so that $F_j=\bigcup_{r\in A_j} G_r$, and for each $1\leq i\leq n$ there exists a $B_i\sub\{1,\ldots,N\}$ so that $E_{k_i}=\bigcup_{r\in B_i} G_r$. Thus,
	\[
	\supp(x)=\bigcup_{j=1}^m F_{j}=\bigcup_{j=1}^m\bigcup_{r\in A_j} G_r \quad\text{and}\quad \supp(x)=\bigcup_{i=1}^n E_{k_i}=\bigcup_{i=1}^n\bigcup_{r\in B_i} G_r.
	\]
	Now observe
	\begin{align}\label{Eq:triplesum}
		 \Big\|\sum_{i=1}^n a_i y_{k_i}\Big\|^p&=\Big\|\sum_{j=1}^m\sum_{r\in A_j} \sum_{l\in G_r} a_{i_r} b_l f_l \Big\|^p=\sum_{j=1}^m\left(\sum_{r\in A_j}\sum_{l\in G_r} |a_{i_r} b_l|^{p_{s_r}}\right)^{p/p_{s_r}}\\
		&=\sum_{j=1}^m\left(\sum_{r\in A_j}|a_{i_r}|^{p_{s_r}}\sum_{l\in G_r} |b_l|^{p_{s_r}}\right)^{p/p_{s_r}}. \notag
	\end{align}
	Note that $\frac{p}{p_{s_r}}\geq 1$ for all $r$ and so $(x+y)^{p/p_{s_r}}\geq x^{p_{s_r}}+y^{p_{s_r}}$ for all nonnegative $x,y\in\R$. Therefore, (\ref{Eq:triplesum}) is
	\begin{align*}
		&\geq \sum_{j=1}^m\sum_{r\in A_j}|a_{i_r}|^p\Big(\sum_{l\in G_r} |b_l|^{p_{s_r}}\Big)^{p/p_{s_r}} =\sum_{i=1}^n\sum_{r\in B_i}|a_{i}|^p\Big(\sum_{l\in G_r} |b_l|^{p_{s_r}}\Big)^{p/p_{s_r}} \\
		&=\sum_{i=1}^n|a_i|^p\sum_{r\in B_i}\Big(\sum_{l\in G_r} |b_l|^{p_{s_r}}\Big)^{p/p_{s_r}} =\sum_{i=1}^n|a_i|^p\quad \text{ by (\ref{Eq:norm1}).}
	\end{align*}
	Next, by the fact that the function $f(x)=x^{p/p_{s_r}}$ is convex and the fact that $|A_j|\leq n$ for all $j$, we know that (\ref{Eq:triplesum}) is
	\begin{align*}
		&\leq \sum_{j=1}^m\sum_{r\in A_j}\frac{\big(|A_j||a_{i_r}|^p\sum_{l\in G_r} |b_l|^{p_{s_r}}\big)^{p/p_{s_r}}}{|A_j|} \\ &\leq n^{\frac{p}{p_{s_0}}-1} \, \sum_{j=1}^m\sum_{r\in A_j}|a_{i_r}|^p\Big(\sum_{l\in G_r} |b_l|^{p_{s_r}}\Big)^{p/p_{s_r}}  \\ &=n^{\frac{p}{p_{s_0}}-1} \,\sum_{i=1}^n\sum_{r\in B_i}|a_{i}|^p\Big(\sum_{l\in G_r} |b_l|^{p_{s_r}}\Big)^{p/p_{s_r}} 
		\\ &=n^{\frac{p}{p_{s_0}}-1} \,\sum_{i=1}^n|a_i|^p\sum_{r\in B_i}\Big(\sum_{l\in G_r} |b_l|^{p_{s_r}}\Big)^{p/p_{s_r}} \\ &=n^{\frac{p}{p_{s_0}}-1} \,\sum_{i=1}^n|a_i|^p\quad \text{ by (\ref{Eq:norm1})} 
	\end{align*}
	which establishes Claim 1.\hfill \qed
	
	\noindent\emph{Claim 2: If the sequences $(p_s)$ and $(n_s)$ are chosen so that 
	\[
	n_s> s^\frac{pp_s}{p-p_s}.
	\]
	for each $s$, then $(f_i)$ is not equivalent to the unit vector basis of any $\ell_p$ or $c_0$. Moreover, $X$ will not embed into any of the spaces $\ell_p$ or $c_0$.}
	
	We show the second, stronger statement. First, note that $X$ contains a subspace isomorphic to $\ell_p$ so clearly $X$ cannot be a subspace of any $\ell_q$, $q\not=p$, or $c_0$. We show that $X$ does not embed into $\ell_p$  by showing that $X$ does not have type $p$. If $X$ had type $p$, then there should exist a $C\geq1$ such that, when the $n_s$ basis vectors from $\ell_{p_s}^{n_s}$ are chosen, the inequalities
	\begin{align*}
	\underset{\e_j=\pm 1}{\operatorname{Ave}}\,\Big\|\sum_{j=1}^{n_s} \e_j e_{s,j}\Big\|&\leq C\left(\sum_{j=1}^{n_s}\|e_{s,j}\|^p\right)^{1/p} \\ &\quad \Longleftrightarrow \quad
	n_s^{1/{p_s}}\leq C n_s^{1/p}\quad \Longleftrightarrow \quad n_s\leq C^{\frac{pp_s}{p-p_s}}
	\end{align*}
	hold. However, this is a contradiction to our choice of $(p_s)$ and $(n_s)$ whenever $s$ is larger than $C$.\hfill \qed
\end{example}

\begin{remark}\label{Rmk:cotype}
The reader may observe that we just as easily could have used
\[
X=\left(\sum_{s=1}^\infty \ell_{q_s}^{n_s}\right)_q
\]
where $2\leq q<\infty$ and $(q_s)\sub(q,\infty]$ is a decreasing sequence of real numbers converging to $q$. Then our argument in Claim 2 would appeal to cotypes rather than types.
\end{remark}

If all of the block basic sequences of a basis $(x_i)$ are good, what then may be concluded in addition to Theorem \ref{Thrm:1asymplp}? The language of asymptotic structure is again useful.

\begin{definition}\label{Def:stabasymlp}
	For $1\leq p\leq \infty$, a space $X$ will be called \emph{stabilized asymptotic} $\ell_p$ provided that $X$ is asymptotic $\ell_p$ and $X$ is a stabilizing subspace for $X$ in the sense of Theorem \ref{Thrm:MilTom}. We analogously define stabilized $C$-asymptotic $\ell_p$.
\end{definition}

The previous definition is equivalent to saying that there exists a $C\geq1$ so that for every $n\in\N$ and $\e>0$ there is an $N=N(n,\e)$ such that for every block sequence $(y_i)_{i=1}^n$ of $n$ vectors with $\supp(y_1)\geq N$ we have $(y_i)\stackrel{C+\e}{\sim}\ell_p^n$.

\begin{theorem}\label{Thrm:stab1asymplp}
	Let $(x_i)$ be a basis for $X$ which has the property that every normalized block basic sequence of $(x_i)$ is good. Let $1\leq p\leq \infty$ be as in Proposition \ref{Prop:mustbelp3}. Then $X$ is stabilized $1$-asymptotic $\ell_p$ with respect to $(x_i)$.
\end{theorem}

\begin{proof}
	Suppose that $X$ is not stabilized $1$-asymptotic $\ell_p$. Then there is an $\e>0$ and $n\in\N$ so that for all $N\in\N$ there exist normalized block vectors $(y_i)_{i=1}^n$ with $\supp(y_1)\geq N$ so that $(y_i)_{i=1}^n$ is not $(1+\e)$-equivalent to the unit vector basis of $\ell_p^n$. Hence, there is a normalized block basic sequence $(y_i)_{i=1}^\infty$ of $(x_i)$ so that for all $k$ the block vectors $(y_i)_{kn+1\leq i\leq k(n+1)}$ are not $(1+\e)$-equivalent to $\ell_p^n$. By assumption, $(y_i)$ is good, and by Proposition \ref{Prop:mustbelp3}, the spreading model generated by $(y_i)$ must be $1$-equivalent to the unit vector basis of $\ell_p$. Evidently, this is a contradiction and so $X$ must be stabilized $1$-asymptotic $\ell_p$.
\end{proof}

Analogously to the usual proof which shows that $1$-asymptotic $\ell_p$ spaces contain almost isometric copies of $\ell_p$, one may show that a space which is stabilized $1$-asymptotic $\ell_p$ is isomorphic to an $\ell_p$ sum of finite dimensional spaces in a very strong way. We omit the proof, but such a result gives the following corollary.

\begin{corollary}\label{Cor:lpsum}
	Let $(x_i)$ be a basis for $X$ which has the property that every normalized block basic sequence of $(x_i)$ is good. Let $1\leq p\leq \infty$ be as in Proposition \ref{Prop:mustbelp3}. Then $X$ is isomorphic to an $\ell_p$ sum of finite dimensional spaces (or $c_0$ in the case $p=\infty$). 
	
	Moreover, for any $\e>0$, the decomposition can be made so that, if $F_k=[x_i]_{n_k\leq i<n_{k+1}}$ for an increasing sequence $(n_k)_{k\geq0}$, then $[x_i]_{i\geq n_1}$ is isomorphic to $\left(\sum_{i=1}^\infty F_i\right)_p$ with constant $1+\e$. 
\end{corollary}

This corollary shows that any space satisfying both Claims 1 and 2 of Example \ref{Ex:allgood} must be essentially the same as the example which was given there.

As was the case in the previous section, the converse of Theorem \ref{Thrm:stab1asymplp} holds, and the proof is very similar.

\begin{theorem}\label{Thrm:converse2}
	Assume that $X$ with basis $(x_i)$ is stabilized $1$-asymptotic $\ell_p$ for some $1\leq p\leq\infty$. Then every normalized block basic sequence of $(x_i)$ is good.
\end{theorem}

\section{Good NCCB Basic Sequences}\label{Sec:NCCB}

We now use the Milliken-Taylor theorem as a means to produce basic sequences so that all of their NCCB basic sequences are good. The next lemma is actually a special case of Theorem 3.3 in a paper of Halbeisen and Odell \cite{HalOde}. In that paper, the authors investigate asymptotic models, another description of asymptotic geometry which generalizes spreading models. We reproduce their proof here to show how it uses the Milliken-Taylor theorem the same way the proof of Theorem \ref{Thrm:spreadingmodel} uses Ramsey's theorem.

\begin{lemma}\label{Lemma:constantcoeff}
	Let $(x_i)$ be a normalized basic sequence in a Banach space $X$. There exists $P\in\langle\N\rangle^\omega$ such that for any $Q=(Q_i)\in\langle P\rangle^\omega$ the NCCB basic sequence
	\[
	y_i=\frac{\sum_{k\in Q_i} x_k}{\|\sum_{k\in Q_i} x_k\|}.
	\]
	corresponding to $(Q_i)$ is a good sequence. Moreover, all such sequences generate spreading models which are $1$-equivalent to each other.
\end{lemma}

\begin{proof}
	 Fix some sequence $(\e_m)$ of numbers such that $0<\e_m<1$ and $\e_m\to 0$. Let $(a_i)_{i=1}^n\in[-1,1]^n$ be scalars. Define the map $f:\langle\N\rangle^n\to[0,n]$ by
	\[
	f(E_1,\ldots,E_n)=\left\|\sum_{i=1}^n a_i \frac{\sum_{k\in E_i} x_k}{\|\sum_{k\in E_i} x_k\|}\right\|.
	\]  
	Since $[0,n]$ is compact, we apply Theorem \ref{Thrm:miltay} to find a $Q(1)\in\langle\N\rangle^\omega$ such that  
	\[
	\left|\,\Big\|\sum_{i=1}^n a_i \frac{\sum_{k\in E^1_i} x_k}{\|\sum_{k\in E^1_i} x_k\|}\Big\|-\Big\|\sum_{i=1}^n a_i \frac{\sum_{k\in F^1_i} x_k}{\|\sum_{k\in F^1_i} x_k\|}\Big\|\,\right|\leq\e_1
	\]
	whenever $(E^1_i),(F^1_i)\in\langle Q(1)\rangle^n$. In particular, this means that whenever  $(y_i^1)$ is a NCCB basic sequence corresponding to a blocking of naturals $(E^1_i)\in\langle Q(1)\rangle^\omega$ then
	\[
	\left|\,\Big\|\sum_{i=1}^n a_i y_{j_i}^1\Big\|-\Big\|\sum_{i=1}^n a_i y_{k_i}^1\Big\|\,\right| \leq\e_1
	\]
	for any $j_1<\ldots<j_n$ and $k_1<\ldots<k_n$. Repeating this process for each $m$, we can find a sequence of blockings $(Q(m))\sub\langle\N\rangle^\omega$ such that $Q(m+1)\ssub Q(m)$  and if $(y_i^m)$ is a NCCB basic sequence corresponding to a blocking of naturals $(E^m_i)\in\langle Q(m)\rangle^\omega$ then
	\[
	\left|\,\Big\|\sum_{i=1}^n a_i y_{j_i}^m\Big\|-\Big\|\sum_{i=1}^n a_i y_{k_i}^m\Big\|\,\right| \leq\e_m
	\]
	for any $j_1<\ldots<j_n$ and $k_1<\ldots<k_n$. If $Q(m)=(Q^m_i)$, then let $Q=(Q^i_i)$ be the diagonal blocking. Then for any blocking of naturals $(E_i)\in\langle Q\rangle^\omega$ and NCCB basic sequence $(y_i)$ corresponding to $(E_i)$ the limit $\lim_{k_1\to\infty} \|\sum_{i=1}^n a_i y_{k_i}\|$ exists. Moreover, by construction, this limit is the same for any $(E_i)$ so chosen. 
	
	Enumerate the countable set
	\[
	R=\{(a_i)_{i=1}^n: n\in\N, \, a_i\in\Q\cap[-1,1] \text{ for } 1\leq i\leq n\}.
	\]
	By the technique of the previous paragraph, we can find blockings of naturals $P(m)\in\langle\N\rangle^\omega$ for all $m\in\N$ such that $P(m+1)\ssub P(m)$ and such that for all $m\in\N$ the limit $\lim_{k_1\to\infty} \|\sum_{i=1}^n a_i y_{k_i}\|$ exists whenever $(a_i)$ is one of the first $m$ elements in the enumeration of $R$ and whenever $(y_i)$ is a NCCB basic sequence corresponding to $(E_i)\in\langle P(m)\rangle^\omega$. Moreover, the value of this limit depends on $(a_i)$ but not on $(E_i)$.  Let $P$ be the diagonal blocking of these $P(m)$ and then the limit $\lim_{k_1\to\infty} \|\sum_{i=1}^n a_i y_{k_i}\|$ exists whenever any $(a_i)$ in $R$ is used and whenever $(y_i)$ is a NCCB basic sequence corresponding to $(E_i)\in\langle P\rangle^\omega$. This is sufficient to show that all NCCB basic sequences $(y_i)$ corresponding to $(E_i)\in\langle P\rangle^\omega$ are good sequences. Since the relevant limits depend on $(a_i)$ but not on $(E_i)$ the spreading models of all these sequences are $1$-equivalent to each another.
\end{proof}

We can now present a new stabilization for the spreading models of NCCB basic sequences.

\begin{theorem}\label{Thrm:newstab}
	Let $(x_i)$ be a normalized basic sequence in a Banach space $X$. There exists a (not necessarily normalized) block basic sequence $(y_i)$ of $(x_i)$ such that all NCCB basic sequences of $(y_i)$ are good sequences and all such sequences have spreading models which are $1$-equivalent.
\end{theorem}

\begin{proof}
	Let $P=(P_i)\in\langle\N\rangle^\omega$ be the blocking of naturals given by Lemma \ref{Lemma:constantcoeff}. Define $y_i=\sum_{k\in P_i} x_k$. Then for every blocking of naturals $(E_i)\in\langle\N\rangle^\omega$ the corresponding NCCB basic sequence $(z_i)$ of $(y_i)$ satisfies
	\[
	z_i=\frac{\sum_{k\in E_i} y_k}{\|\sum_{k\in E_i} y_k\|} =\frac{\sum_{k\in E_i} \sum_{l\in P_k} x_l}{\|\sum_{k\in E_i} \sum_{l\in P_k} x_l\|}=\frac{\sum_{l\in F_i} x_l}{\|\sum_{l\in F_i} x_l\|}
	\] 
	where $(F_i)\in\langle P\rangle^\omega$. Hence, by Lemma \ref{Lemma:constantcoeff}, all such $(z_i)$ are good and have the same spreading model.
\end{proof}

\begin{remark}\label{Rmk:allequiv}
	It is no coincidence that when we were able to stabilize all of the NCCB basic sequences of $(x_i)$ we also found that all of the spreading models generated by these sequences were $1$-equivalent to one another. Actually, this a necessary consequence of the fact that all NCCB basic sequences are good. To see this, suppose that all NCCB basic sequences of $(x_i)$ are good but that there are two such sequences, say $(y^1_i)$ and $(y^2_i)$ which generate spreading models which are not $1$-equivalent. Then there exists some NCCB basic sequence $(z_i)$ of $(x_i)$ such that $(z_{2k+1})$ is a subsequence of $(y^1_i)$ and $(z_{2k})$ is a subsequence of $(y^2_i)$. Then clearly $(z_k)$ cannot be good.
\end{remark}

The sequence $(y_i)$ obtained in Theorem \ref{Thrm:newstab} is not necessarily normalized nor is it likely to be even semi-normalized. However, if it is semi-normalized and also unconditional, then the following shows that the spreading model must be equivalent to $c_0$ or $\ell_p$. The proof will appeal to Zippin's theorem from \cite{Zip}. Zippin's theorem states that, if $(x_i)$ is a normalized unconditional basis which is equivalent to all of its NCCB basic sequences, then $(x_i)$ is equivalent to $c_0$ or to $\ell_p$ for some $1\leq p<\infty$.

\begin{theorem}\label{Thrm:asymptoticzippin}
	Let $(x_i)$ be an unconditional semi-normalized basic sequence in a Banach space such that every NCCB basic sequence of $(x_i)$ is good. Then all the spreading models generated by these NCCB basic sequences must be uniformly equivalent to the unit vector basis of $c_0$ or $\ell_p$ for some $1\leq p<\infty$.
\end{theorem}

\begin{proof}
	Let $(x_i)$ be as stated and $(z_i)$ be its normalization. Let $U$ be the unconditionality constant of $(x_i)$ and $\d$ and $M$ be such that $0<\d\leq\|x_i\|\leq M<\infty$. Assume that $(e_i)$ is the spreading model of $(z_i)$. Let $(F_i)\in\langle \N\rangle^\omega$ and $(f_i)$ be the NCCB basic sequence of $(e_i)$ corresponding to $(F_i)$. We will show that $(f_i)$ is equivalent to $(e_i)$. 
	
	Fix $n\in\N$ and scalars $(a_i)_{i=1}^n$. Define a NCCB basic sequence $(y_i)$ of $(x_i)$ corresponding to any $(E_i)\in\langle\N\rangle^\omega$ satisfying $|E_{kn+i}|=|F_i|$ where $k$ is a nonnegative integer and $1\leq i\leq n$. In other words, $(E_i)$ is some blocking of naturals where the cardinalities of the blocks repeat the cardinalities of the first $n$ elements of $(F_i)$. Then by assumption $(y_i)$ is good and generates a spreading model $(s_i)$ which is $1$-equivalent to $(e_i)$ by Remark \ref{Rmk:allequiv}.
	
	Since $(e_i)$ is the spreading model of $(z_i)$, it is possible to choose $k$ large enough that
	\begin{align*}
	\frac{1}{2}\Bigg\|\sum_{i=1}^n \sum_{l\in F_i} \frac{a_i}{\|\sum_{l\in F_i} e_l \|} e_l\Bigg\|&\leq\Bigg\|\sum_{i=1}^n \sum_{l\in E_{kn+i}} \frac{a_i}{\|\sum_{l\in F_i} e_l \|} z_l\Bigg\| \\
	&\leq 2\Bigg\|\sum_{i=1}^n \sum_{l\in F_i} \frac{a_i}{\|\sum_{l\in F_i} e_l \|} e_l\Bigg\|
	\end{align*}
	and 
	\[
	\frac{1}{2}\Big\|\sum_{l\in F_i} e_l\Big\|\leq \Big\|\sum_{l\in E_{kn+i}} z_l\Big\|\leq 2\Big\|\sum_{l\in F_i} e_l\Big\|
	\]
	for all $1\leq i\leq n$. Also, since $(s_i)$ is the spreading model of $(y_i)$ it is possible to choose $k$ perhaps larger so that
	\[
	\frac{1}{2}\Big\| \sum_{i=1}^n a_i s_i\Big\|\leq \Big\|\sum_{i=1}^n a_i y_{kn+i}\Big\|\leq 2\Big\| \sum_{i=1}^n a_i s_i\Big\|.
	\]
	Under these assumptions, we have 
	\[
	\Big\|\sum_{l\in F_i} e_l \Big\|\leq 2\Big\|\sum_{l\in E_{kn+i}}z_l\Big\|= 2\Big\|\sum_{l\in E_{kn+i}}\frac{x_l}{\|x_l\|}\Big\|\leq \frac{2U}{\d}\Big\|\sum_{l\in E_{kn+i}} x_l\Big\|
	\]
	for $1\leq i\leq n$ and so 
	\[
	\frac{|a_i|}{\|\sum_{l\in F_i} e_l\|\, \|x_l\|}\geq \frac{\d|a_i|}{2UM\|\sum_{l\in E_{kn+i}} x_l\|}.
	\]
	Similarly,
	\[
	\Big\|\sum_{l\in F_i} e_l \Big\|\geq \frac{1}{2UM}\Big\|\sum_{l\in E_{kn+i}} x_l\Big\|\quad\text{ and }\quad \frac{|a_i|}{\|\sum_{l\in F_i} e_l\|\,\|x_l\|}\leq \frac{2UM|a_i|}{\d\|\sum_{l\in E_{kn+i}} x_l\|}.
	\]
	Combining all this together, we obtain
	\begin{align*}
		\Big\|&\sum_{i=1}^n a_i f_i\Big\|=\Bigg\|\sum_{i=1}^n \sum_{l\in F_i} \frac{a_i}{\|\sum_{l\in F_i} e_l \|} e_l\Bigg\|\leq 2\Bigg\|\sum_{i=1}^n \sum_{l\in E_{kn+i}} \frac{a_i}{\|\sum_{l\in F_i} e_l \|} z_l\Bigg\| \\
		&=2\Bigg\|\sum_{i=1}^n \sum_{l\in E_{kn+i}} \frac{a_i}{\|\sum_{l\in F_i} e_l \|\, \|x_l\|} x_l\Bigg\|\leq \frac{4U^2M}{\d} \Big\|\sum_{i=1}^n a_i \frac{\sum_{l\in E_{kn+i}} x_l}{\|\sum_{l\in E_{kn+i}} x_l\|}\Big\| \\
		&=\frac{4U^2M}{\d} \Big\|\sum_{i=1}^n a_i y_{kn+i}\Big\|\leq \frac{8U^2M}{\d} \Big\|\sum_{i=1}^n a_i s_i\Big\|=\frac{8U^2M}{\d}\Big\|\sum_{i=1}^n a_i e_i\Big\|,
	\end{align*}
	and similarly
	\[
	\Big\|\sum_{i=1}^n a_i f_i\Big\|\geq\frac{\d}{8U^2 M }\Big\|\sum_{i=1}^n a_i e_i\Big\|.
	\]
	Thus, $(f_i)$ is equivalent to $(e_i)$. The result follows from Zippin's theorem if we show that $(e_i)$ is unconditional. This is true because $(x_i)$, being semi-normalized and unconditional, is equivalent to its normalization $(z_i)$. Hence, $(z_i)$ is unconditional and so also is its spreading model $(e_i)$.
\end{proof}

\begin{remark}
	In their paper \cite{OdeSch3}, Odell and Schlumprecht constructed a space with unconditional basis which has no $c_0$ or $\ell_p$ spreading model. Consequently, Theorem \ref{Thrm:asymptoticzippin} shows that one cannot do better than our result in Theorem \ref{Thrm:newstab}. There is no Ramsey-theoretic stabilization for a general Banach space $X$ that will guarantee a semi-normalized basic sequence all of whose NCCB basic sequences are good.
\end{remark}

\begin{remark}
  In Theorem \ref{Thrm:asymptoticzippin}, the assumption of unconditionality is necessary. To see this, consider the James space of \cite{Jam} which has no unconditional basis. While the standard unit vector basis $(e_i)$ for this space is not spreading, the basis $(s_k)$ where $s_k=\sum_{i=1}^k e_i$ is $1$-spreading. In fact, every NCCB basic sequence of $(s_k)$ is $1$-spreading. Thus, all NCCB basic sequences are trivially good and generate spreading models. These spreading models are equivalent to $(s_k)$ and not the unit vector bases of $c_0$ or $\ell_p$. Therefore, the conclusion of Theorem \ref{Thrm:asymptoticzippin} fails.
\end{remark}

\subsection*{Acknowledgements}
The results of this article were obtained during the preparation of the author's PhD dissertation. This work was undertaking at the University of North Texas under the supervision of B\"unyamin Sar\i\,\,to whom the author would like to give thanks for many helpful conversations.

%%%%%%%%%%% To ease editing, use normal size for the references:

\end{document}